\documentclass[11pt, a4paper]{article}
\usepackage{fullpage}
\usepackage[utf8]{inputenc}
\usepackage{footnote}
\usepackage{amsmath, amsfonts, amsthm, mathtools, xcolor}
\usepackage{array}
\usepackage{systeme}
\usepackage{bm}
\usepackage{upgreek}
\usepackage[english]{babel}
\usepackage{enumitem}

\newcommand{\setbuilder}[2]{\left\{#1\ \colon #2\right\}}

\newtheorem{problem}{Problem}
\newtheorem{lemma}{Lemma}
\newtheorem{theorem}{Theorem}

\newtheorem{proposition}{Proposition}

\theoremstyle{definition}

\usepackage[T2A]{fontenc}

\title{On some non-rigid unit distance patterns}
\author{Nóra Frankl\thanks{Alfréd Rényi Institute of Mathematics, Budapest, Email: {\tt nfrankl@renyi.hu} 
} \and Dora Woodruff\thanks{Harvard University, Cambridge, MA, Email: {\tt dorawoodruff@college.harvard.edu} 
} }
\date{}

\begin{document}

\maketitle

\begin{abstract} A recent generalization of the Erdős Unit Distance Problem, proposed by Palsson, Senger and Sheffer, asks for the maximum number of unit distance paths with a given number of vertices in the plane and in $3$-space. Studying a variant of this question, we prove sharp bounds on the number of unit distance paths and cycles on the sphere of radius $1/\sqrt{2}$. We also consider a similar problem about $3$-regular unit distance graphs in $\mathbb{R}^3$.
\end{abstract}

\section{Introduction}
\par The Erd\H{o}s Unit Distance Problem is one of the most famous unsolved problems in discrete geometry. It asks for $u_2(n)$, the maximum possible number of unit distances among $n$ points in the plane. The best known lower bound, $u_2(n)=\Omega(n^{1+c/\log \log n})$ for some constant $c$ is due to Erd\H{o}s \cite{erdos1946sets}, and the current best upper bound $u_2(n)=O(n^{\frac{4}{3}})$ is due to Spencer, Szemerédi and Trotter\cite{4/3}. The analogous problem is also interesting, and still far from a resolution, in $\mathbb{R}^3$ and on spheres of most radii. We note that however, starting from dimension $4$, up to the order of magnitude, the question is less interesting, as one can find $\Omega(n^2)$ many unit distances by a well known construction of Lentz \cite{Lenz}. The exact value for large $n$ and even $d\geq 4$ was determined by Brass \cite{Brass} and Swanepoel \cite{Swanepoel}.

Several variants and generalizations of the unit distance problem have been studied. Recently, Sheffer, Palsson and Senger \cite{discretechains} proposed to find the maximum number of unit distance paths $P_k(n)$ with $k$ vertices in $\mathbb{R}^2$ or $\mathbb{R}^3$.  This problem was essentially solved in $\mathbb{R}^2$ in \cite{almostsharp} by finding almost sharp bounds for $k= 0,1 \mod 3$ and showing that the for $k= 2 \mod 3$ the problem is essentially equivalent to the Unit Distance problem. Similar questions have been studied for paths and trees determined by dot products by Kilmer, Marshall, Senger \cite{KMS} and Gunter, Palsson, Rhodes, Senger \cite{GPRS}. Passant \cite{Passant} obtained results for the corresponding distinct distances problem.

\par We continue this line of research. First, we study the unit distance path problem on the sphere $\mathbb{S}^2$ of radius $\frac{1}{\sqrt{2}}$. Stereographic projection from the center of the sphere shows that unit distance graphs on a sphere of this radius are very similar to point-line incidence graphs in the plane. The only difference comes from the fact that when we project a line to a great circle on the sphere, we can choose any of its poles to represent the line. Thus, while the maximum number of edges in the two graphs are in within a constant factor, $K_{2,t}$ subgraphs with $t\geq 3$ are not excluded from unit distance graphs on the sphere. In other words, instead of unit distance graphs on the sphere, we could think about point-line incidence graphs in the plane, but allowing every line to be used twice.

For any fixed $k$, we determine $P^S_k(n)$, the maximum number of unit distance paths with $k$ vertices on the sphere $\mathbb{S}^2$, up to a polylogarithmic factor. We will use notations $\tilde{\Theta }$ and $\tilde{O}$ to hide a poly-logarithmic error term.

\begin{theorem}\label{thm path}
For any fixed $k\geq 1$,  the number of paths on $k$ vertices on the sphere is
\[
P^S_k(n)= 
\begin{cases}
\tilde{\Theta}\left(n^{\lfloor 2(k+3)/5 \rfloor}\right),
& \mbox{if } k = 0,1,3,4 \textrm{ } ({\rm mod\ } 5),\\[4pt]
\tilde{\Theta}\left (n^{\lfloor 2(k+3)/5 \rfloor-2/3}\right),
& \mbox{if } k = 2\textrm { } ({\rm mod\ } 5).
\end{cases}\]
\end{theorem}

The $k=2$ case of the theorem above (without the polylogarithmic error), via stereographic projection, is the Szemerédi-Trotter bound for point-line incidences. Note that while the planar quantity, $P_k(n)$ depends on $k$ mod $3$, on the sphere the answer depends on $k$ mod $5$. The constructions for the lower bounds, described in Section $2$, will explain this difference. We also remark that the exponent $2k/5$ is very close the non-tight upper bounds from \cite{discretechains} for the planar case, but this appears to be a coincidence.

Next, we study $C_k^S(n)$, the maximum possible number of unit distance cycles with $k$ vertices determined by a set of $n$ points in the sphere $\mathbb{S}^2$. For most cycle lengths we have almost sharp results, however working with short cycles is more difficult. Again, bounding the number of cycles of length $2k$ on the sphere is equivalent to finding a bound on the number of polygons with $k$ vertices in the plane that can be determined by $n$ points and $n$ lines, such that we are allowed to use every line twice.

\begin{theorem}\label{thm cycles}
We have $C_4^S(n)=\Theta(n^2)$, and for any $k\geq 5$, with the exception of $k = 6, 7, 9$ we have
\[
C^S_k(n)= 
\begin{cases} \tilde{\Theta}\left (n^{\lfloor 2k/5 \rfloor}\right ),
& \mbox{if } k = 0,1,3,4 ({\rm mod\ } 5),\\[4pt]
\tilde{\Theta}\left (n^{\lfloor 2k/5 \rfloor+1/3}\right ),
& \mbox{if } k = 2 ({\rm mod\ } 5).
\end{cases}\]
\end{theorem}

For $k=3,6,7,9$ there is a gap between the exponent of the lower and upper bounds. We summarize the (to our knowledge) best bounds for these lengths in Proposition \ref{prop small k} in Section $3$. We note that a related problem about cycles in incidence graphs of points and lines was studied by  de Caen and Székely \cite{Szekely}. They conjectured that the maximum number of $6$-cycles determined by an incidence graph of $n$ points and $m$ lines is $O(mn)$, which was disproved by Klavík, Král and Mach \cite{Kral}.

Next, we turn to a similar question in $\mathbb{R}^3$. We study the maximum number of unit distance subgraphs isomorphic to a given $3$-regular graph $G$.

\begin{theorem}\label{thm 3-reg}Let $G$ be a fixed $3$-regular graph on $k$ vertices. The maximum number of unit distance subragraphs isomorphic to $G$ determined by a set of $n$ points in $\mathbb{R}^3$ is $\tilde{O}\left(n^{k/2}\right)$.
\end{theorem}

By slightly modifying the problem, and asking for the maximum number of copies of $G$ with prescribed edge lengths, our upper bound remains valid. In this modified setting, for bipartite graphs we can match this bound by simple constructions.

\subsubsection*{Acknowledgment}This project was done as part of the 2021 New York Discrete Math REU, funded by NSF grant DMS 2051026. NF was partially supported by ERC Advanced Grant "GeoScape". We thank Adam Sheffer and Pablo Soberón for their organization of the REU, as well as all mentors and participants of the program for their support.

\section{Paths on the sphere}\label{sec path}

We begin by recalling the Szemerédi-Trotter bound \cite{SzT} on the number of point line incidences. For the maximum number of incidences $I(n,m)$ between a set of $n$ points and $m$ lines in the plane we have
\begin{equation}\label{lines}
    I(n,m) = \Theta\left (n^{\frac{2}{3}}m^{\frac{2}{3}} + m + n\right ).
\end{equation}

Let $u(m,n)$ denote the maximum number of unit distance pairs between a set of $n$ and a set of $m$ points on the sphere. Via stereographic projection, \eqref{lines} implies  
\begin{equation}\label{SzT}
    u(n,m) = \Theta\left (n^{\frac{2}{3}}m^{\frac{2}{3}} + m + n\right ).
\end{equation}

For any $r \geq 1$, we say that a point $p$ on  $\mathbb{S}^2$ is \emph{$r$-rich} with respect to a set of $n$ points $P\subseteq \mathbb{S}^2$, if it is unit distance apart from at least $r$ points of $P$. A well-known equivalent formulation of \eqref{lines} gives that the maximum number of $r$-rich points with respect to $P$ is
\begin{equation}\label{rich}
    O\left (\frac{n^3}{r^2} + \frac{n}{r}\right ).
\end{equation}

We now list some simple observations, which will be very helpful in this and in the following section. First, on a sphere of radius $\frac{1}{\sqrt{2}}$, two points $p, q$ are unit distance apart if and only if $p$ lies on the great circle that has $q$ as a pole. Furthermore, if $p$ and $q$ are not antipodal, then there are at most $2$ points unit distance from both of them. However, if $p$ and $q$ are antipodal, then any point lying on their great circle will be unit distance from both. The work of Palsson, Senger, and Sheffer in \cite{discretechains} and Frankl and Kupavskii in \cite{almostsharp} for paths in the plane relies on the fact that in the plane, there are at most two points unit distance from two fixed points. Therefore, we will need to find some way to work around the existence of antipodal pairs in our proofs.

We call a path $(p_1,p_2,\dots,p_k)$ on the sphere \emph{antipodal-free} if there is no $1\leq i \leq k-2$ such that $p_i$ and $p_{i+2}$ are antipodal. From the observations above, Theorem 2 in \cite{almostsharp} implies the following statement. 

\begin{proposition}\label{prop planar} For any fixed $k$, the number of antipodal-free $k$-paths $(p_1,p_2,\dots,p_k)$ determined by a set of $n$ points on the sphere is at most $P_k(n)$, the  number of $k$-paths in a set of $n$ points in the plane. That is, the number of antipodal-free $k$-paths is $\tilde{O}(n^{\lfloor k/3\rfloor +1})$ for $k= 0,1 \mod 3$, and $\tilde{O}(n^{ (k+2)/3})$ for $k= 2 \mod 3$.
\end{proposition}

\begin{proof}[Proof of Theorem \ref{thm path}] We start by proving the upper bounds. The proof is by induction on $k$. We have to consider several base cases.
\begin{itemize}
\item For $k=1$, we trivially have $P^S_1(n) = n$. 
\item For $k=2$, by \eqref{SzT} we have $P^S_2(n) = O(n^{\frac{4}{3}})$. 
\item For $k=3$, in a path $(p_1,p_2,p_3)$ either $p_1$ and $p_3$ are antipodal, or not. In the first case, after choosing $p_1$, the antipodal pair $p_3$ is uniquely determined, giving the bound $O(n^2)$. In the second case, after choosing $p_1$ and $p_3$, there are at most two choices for the middle vertex $p_2$, and we obtain again the bound $O(n^2)$. Overall we still obtain $P^S_3(n) =O(n^2)$.
\item For $k=4$, any path $(p_1,p_2,p_3,p_4)$ is either antipodal-free, or not. If it is not antipodal-free, we may assume without loss of generality that $p_2$ and $p_4$ are antipodal. Then after choosing $(p_1,p_2,p_3)$, the last vertex $p_4$ is uniquely determined. By the $k=3$ case we have $O(n^2)$ choices for $(p_2,p_3,p_4)$, obtaining the $O(n^2)$ bound.
In the antipodal-free case Proposition \ref{prop planar} implies the bound $\tilde{O}(n^2)$. Adding together the two cases, we obtain the bound $\tilde{O}(n^2)$.

\item For $k=5,6,8$, in any path $(p_1,\dots,p_k)$ either $p_{k-2}$ and $p_k$ are antipodal, or not. In the first case, after choosing $(p_1,\dots,p_{k_1})$ the last vertex $p_k$ is uniquely determined, and we are done by the $k-1$ case. In the second case, after choosing $(p_1,\dots,p_{k-2})$ and $p_k$, there are at most $2$ options for $p_{k-1}$, and we are done by the $k-2$ case.
\item For $k=7$, in any path $(p_1,\dots,p_7)$ at most one of $(p_2,p_4)$ and $(p_4,p_6)$ are antipodal (if both pairs were antipodal, then we would have $p_2 = p_6$, which is forbidden). Without loss of generality, we may assume that $p_4$ and $p_6$ are not antipodal. Then after choosing $(p_1,p_2,p_3,p_4)$ and $(p_6,p_7)$ we have at most two options for $p_5$ and we obtain the bound $\tilde{O}(n^2)O(n^{4/3})=\tilde{O}(n^{10/3})$ bound by the $k=4$ and $k=2$ cases.
\item For $k=9$, in any path $(p_1,\dots,p_9)$ either $p_4$ and $p_6$ are antipodal, or not. If they are not antipodal, then after choosing $(p_1,\dots,p_4)$ and $(p_6,\dots,p_9)$ we have at most $2$ choices for $p_5$, and obtain the $\tilde{O}(n^4)$ bound by the $k=4$ case. If $p_4$ and $p_6$ are antipodal, then $p_6$ and $p_8$ cannot be antipodal. Then after choosing $(p_1,p_2)$, $(p_5,p_6)$ and $(p_8,p_9)$, the vertex $p_4$ is uniquely determined, and we have at most two choices for $p_3$ and $p_7$. Thus, we obtain the bound $O(n^{4/3})O(n^{4/3})O(n^{4/3})=O(n^{4})$ bound by the $k=2$ case.
\end{itemize}

For the induction step, we notice that in the bounds we want to prove
\begin{equation}\label{k-5}
\textrm{the difference between the exponent of } P^S_k(n) \textrm{ and } P^S_{k-5}(n) \textrm{ is } 2 \textrm{ for any } k\geq 6
\end{equation}
and
\begin{equation}\label{k-8}
\textrm{the difference between the exponent of } P^S_k(n) \textrm{ and } P^S_{k-8}(n) \textrm{ is at least } 3 \textrm{ for any } k\geq 9.
\end{equation}

In any path $(p_1,\dots,p_k)$ either one of the pairs $(p_4,p_6), (p_{k-3}, p_{k-5})$ are antipodal, or none of them are antipodal. We bound the number of each of these type of paths separately.

\par If $p_4$ and $p_6$ are not antipodal,  then by the $k=4$ case there are $\tilde{O}(n^2)$ different ways to choose $(p_1,p_2,p_3,p_4)$. Further, by definition there are $P^S_{k-5}(n)$ ways to choose $(p_6,\dots,p_k)$. Since $p_4$ and $p_6$ are not antipodal, after choosing the first $4$ and the last $k-5$ vertices, there are at most $2$ different ways to extend it to a path.
This gives the bound 
\begin{equation}\label{first}
\tilde{O}(n^2)P^S_{k-5}(n)
\end{equation}for the number of paths of this type. So, by observation \eqref{k-5} about the exponent of $P^S_{k-5}(n)$ and by induction, in this case we are done.
Symmetrically, if the $p_{k-3}$ and $p_{k-5}$ are antipodal, we obtain again the bound 
\begin{equation}\label{second}
\tilde{O}(n^2)P^S_{k-5}(n).
\end{equation}

\par If both $p_4,p_6$ and $p_{k-3},p_{k-5}$ are antipodal and $k\geq 10$ then after choosing $(p_5,p_6,\dots,p_{k-4})$ the vertices $p_4$ and $p_{k-3}$ are uniquely determined. Further, both $(p_1,p_2)$ and $(p_{k-1},p_k)$ can be chosen in $O(n^{\frac{4}{3}})$ different ways. Since $(p_2,p_4)$ and $(p_{k-3},p_{k-1})$ cannot be antipodal, there are at most two different choices of $p_3$ through which $p_2$ and $p_4$ can be connected, and at most two different choices of $p_{k-2}$ through which $p_{k-3}$ and $p_{k-1}$ can be connected. Together, these imply that the maximum number of paths of this type is 
\begin{equation}\label{third}
O(n^{\frac{4}{3}})P^S_{k-8}(n)O(n^{\frac{4}{3}})=O(n^3)P^S_{k-8}(n).
\end{equation}

From \eqref{first}-\eqref{third} we obtain that the maximum number of $k$-paths is bounded by
\begin{equation*}
\tilde{O}(n^2)P^S_{k-5}(n)+O(n^3)P^S_{k-8}(n).    
\end{equation*}
This, by induction and by observations \eqref{k-5} and \eqref{k-8} about the exponent of $P^S_{k-5}(n)$ and $P^S_{k-8}(n)$ finishes the proof of the upper bound.

\bigskip

We know turn to the lower bound. For the $k= 0,1,3, 4$  $\mod 5$ cases, we imitate the planar constructions from \cite{discretechains}, taking advantage of the antipodal vertices. For an illustration see \mbox{Figure \ref{circles}.} Let $m = \lfloor 5n/2k \rfloor$.
We take $\lceil 2k/5 \rceil$ great circles $K_0,\dots,K_{\lfloor 2k/5-1 \rfloor}$ and on each of them we place a set $Q_i$ of $(m-2)$ points on each such that:
\begin{itemize}
    \item For any $0 \leq i < \lceil 2k/5 \rceil-1$ and for any $p_i\in Q_i$ there is a point $p_{i+1}\in Q_{i+1}$ at unit distance apart from $p_i$.
    \item For any $0 \leq i \leq \lceil 2k/5 \rceil-1$ the set $Q_i$ does not contain any pole of any circle $K_j$.
\end{itemize}

Further, for every $i$ we place two points $N_i$ and $S_i$ in the poles the circle $K_i$. In this construction, we can find $\Omega(n^{\lfloor 2(k+3)/5 \rfloor})$ many $k$-paths $(p_1,p_2,\dots,p_k)$ such that for $i=5\ell+j$ with $1\leq j \leq 5$:
\begin{itemize}
    \item $p_i$ is in $Q_{\ell}$ if $j=1,3,5$
    \item $p_i=N_\ell$ for $j=2$
    \item $p_i=S_{\ell}$ for $j=4$.
\end{itemize}
Indeed, after choosing $p_1$ from $Q_0$ and $p_{5\ell+3}, p_{5\ell+5}$ (for $5\ell+3,5\ell+5\leq k$) from $Q_{\ell}$ arbitrarily for every $\ell$, we can extend the resulting set to a $k$-path.

\medskip

Finally, we explain the modification to obtain the $n^{1/3}$ improvement for $k=2 \mod 5$. We take the construction described previously with the circles $K_i$, points sets $Q_i\in K_i$ and poles $N_i,S_i$ for $k-2$. Take another point set $Q$ of of $m$ points and with $\Omega(m^{4/3})=\Omega(n^{4/3})$ unit distance pairs (this can be done by the same stereographic projection argument discussed in the introduction). Then modify $Q_0$ on $K_5$ such that for any point $q\in Q$ there is a $q_1\in Q_0$ at unit distance apart from $q$. Similarly as before, we can find $\Omega(n^{\lfloor 2(k+3)/5 \rfloor+1/3})$ many $k$-paths $(q_1,q_2,p_1,p_2,\dots,p_{k-2})$ such that for $i=5\ell+j$ with $1\leq j \leq 5$:
\begin{itemize}
    \item $p_i$ is in $Q_{\ell}$ if $j=1,3,5$,
    \item $p_i=N_\ell$ for $j=2$,
    \item $p_i=S_{\ell}$ for $j=4$,
    \item $q_1,q_2\in Q$.
\end{itemize}
Indeed, after choosing a unit distance pair $(q_1,q_2)$ from $Q$, and $p_{5\ell+3}, p_{5\ell+5}$ (for $5\ell+3,5\ell+5\leq k-2$) from $Q_{\ell}$ arbitrarily for every $\ell$, we can extend the resulting set to a $k$-path.

\begin{figure}
    \centering
   {\includegraphics[scale = 0.45]{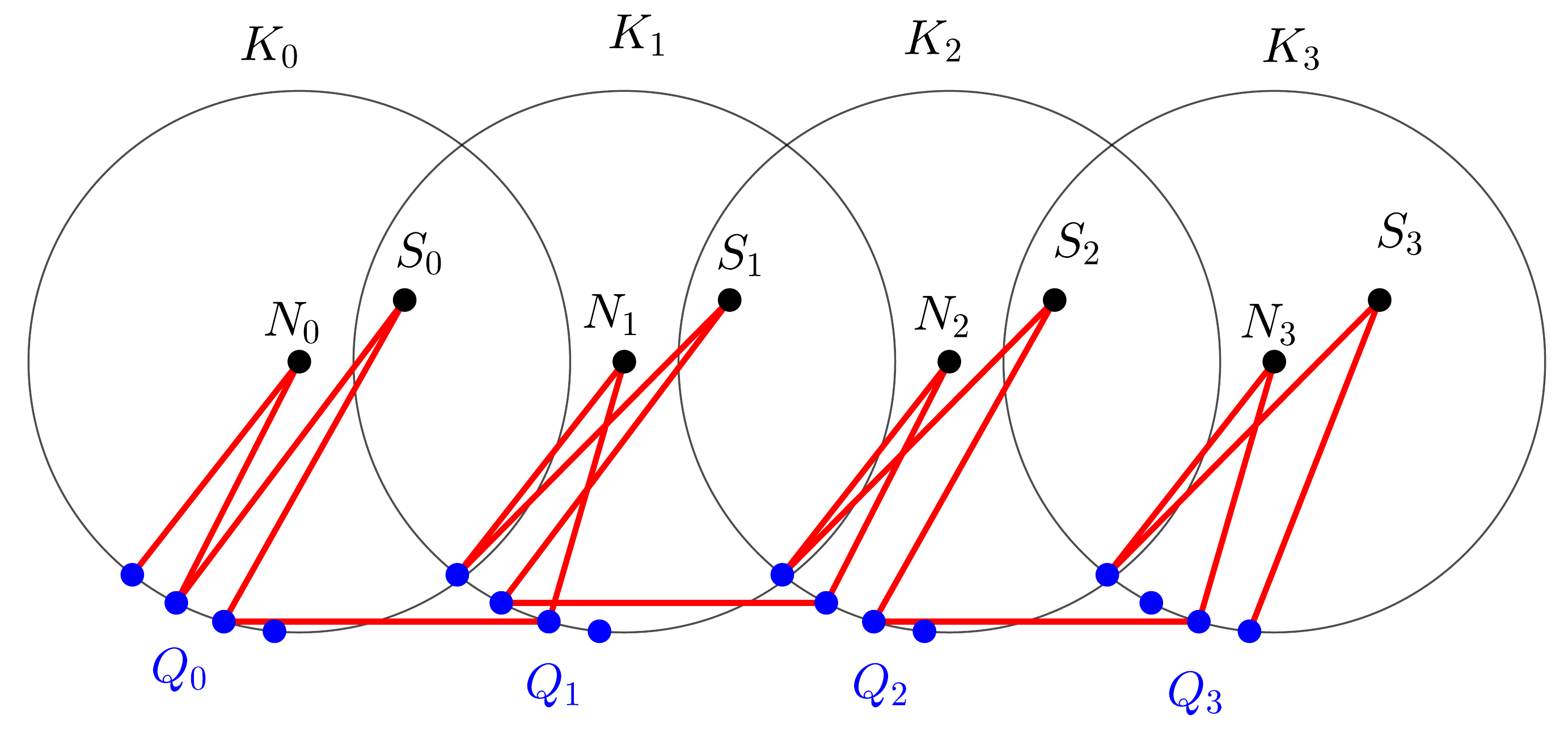}}
    \caption{The circles are the great circles $K_i$ and $N_i$, $S_i$ are their two poles. A possible path starting with a point from $Q_0$ and ending in a point in $Q_3$ is shown in red.}
    \label{circles}
\end{figure}

\end{proof}

\section{Cycles on the sphere of Radius $\frac{1}{\sqrt{2}}$}

To obtain the lower bounds, we slightly modify the path construction as follows. We can arrange the points on the last great circle $K_{\lfloor 2k/5-1 \rfloor}$ such it has a point at unit distance apart from any point of $Q_1$ in the $k= 0,1,3,4$ cases, and from any point of $Q$ in the $k= 2$ case. This will allow closing any $(k-1)$-path $(p_1,p_2,\dots,p_{k-1})$ (in the $k= 0,1,3,4$ cases) or $(q_1,q_2,p_1,p_2,\dots,p_{k-3})$ (in the $k= 2$ case) to a $k$-cycle. Notice that in this construction the exponent is one smaller than the corresponding number of paths for each $k$, thus it indeed matches the claimed bound.

\begin{proof}[Proof of upper bound in Theorem \ref{thm cycles}] We begin by proving the upper bounds. First we bound the number of those cycles $(p_1,p_2,\dots,p_k)$ in which there are at most one antipodal pair separated by one other vertex. If $(p_1,\dots,p_k)$ is antipodal-free, then after choosing $(p_1,\dots,p_{k-1})$, we have at most two choices for the last vertex $p_k$. Further, by Proposition \ref{prop planar}, the number of $(k-1)$-paths $(p_1,\dots,p_{k-1})$ is bounded by $2P_{k-1}(n)$.

If there is exactly one antipodal pair, say $p_1$ and $p_3$, then after choosing the $(k-2)$-path $(p_2,p_3,\dots,p_{k-1})$, the antipodal pair $p_1$ of $p_3$ is uniquely determined. Further, we have at most $2$ choices for the last vertex $p_k$. Thus, Proposition \ref{prop planar}, the number of cycles is bounded by $2P_{k-2}(n)$.

Overall, we obtain that the number of such cycles is bounded by $2P_{k-1}(n)+kP_{k-2}(n)\leq 2kP_{k-1}(n)$. This, by Proposition \ref{prop planar} implies the bound 
$\tilde{O}(n^{\lfloor k/3\rfloor +1})$ for $k= 0,1$ $\mod 3$, and $\tilde{O}(n^{\lfloor k/3\rfloor +1/3})$ for $k= 2$ $\mod 3$. These bounds imply directly the desired bounds. Indeed, for $k\geq 21$ it follows from $k/3+1\leq 2(k+3)/5-1$, and for $k\leq 20$ it can be checked (except for $k=3,6,7,9$) by a brief case analysis.

Thus, we only have to bound the number of cycles under the assumption that there are at least $2$ antipodal pairs. The argument depends on the length of the path up to equivalence$\mod 5$.

\medskip

\noindent $\boldsymbol{k = 0,1,3 \mod 5: }$
In this case we only need the assumption that there is at least one pair of antipodal vertices. If $p_1$ and $p_{k-1}$ are antipodal, then after choosing a $(k-2)$-path $(p_3,p_4,\dots,p_k)$, the vertex $p_1$ is uniquely determined. Further, since $p_1$ and $p_3$ cannot be antipodal, there are at most $2$ choices of $p_2$ to extend $(p_1,p_3,p_4,\dots,p_k)$ to a cycle. Thus the number of cycles in this case is at most $2$-times the number of $(k-2)$-paths, which is $\tilde O(n^{\lfloor (2k+1)/5 \rfloor})=\tilde{O}(n^{\lfloor 2k/5 \rfloor})$ by Theorem \ref{thm path}. As we can argue similarly for any other antipodal pair, overall we obtain the bound $k\tilde O(n^{\lfloor 2k/5 \rfloor })$ for the number of cycles of this type.

\medskip

\noindent $\boldsymbol{k= 2 \mod 5: }$ 
Assume that there are two antipodal pairs $(p_1,p_{3})$ and $(p_i,p_{i+2})$ such that $i\notin \{1,3\}$. First, we prove the bound in the case when the $5$-paths $(p_k,p_1,p_2,p_3,p_4)$ and $(p_{i-1},p_i,p_{i+1},p_{i+2},p_{i+3})$ are disjoint and their complements consist of two non-empty paths $(p_5,\dots,p_{i-2})$ and $(p_{i+4},\dots,p_{k-1})$ of lengths $k_1$ and $k_2$ respectively. Since $k_1+k_2= 2 \mod 5$, we may assume without loss of generality that $(k_1,k_2)= (1,1), (2,0)$ or $(3,4)$ $\mod 5$.
\begin{itemize}
    \item If $(k_1,k_2)=(1,1)$ then after choosing $(p_2,p_3,\dots,p_{i-2})$ and $(p_{i+1},p_{i+2},\dots,p_{k-1})$ arbitrarily, the antipodal pair $p_1$ of $p_3$ and $p_{i}$ of $p_{i+2}$ is uniquely determined. Further, there are at most $2$ choices for $p_k$ and $p_{i-1}$. Thus, the total number of cycles of this type, using Theorem \ref{thm path}, is bounded by
    \begin{multline*}P^S_{k_1+3}(n)\cdot P^S_{k_2+3}(n)=\tilde{O}(n^{\lfloor2(k_1+6)/5 \rfloor+\lfloor 2(k_2+6)/5 \rfloor}) \\
    =\tilde{O}(n^{2(k_1+4)/5+2(k_2+4)/5})=\tilde{O}(n^{(2k-4)/5})=\tilde{O}(n^{\lfloor 2k/5 \rfloor+1/3}).
    \end{multline*}
    \item If $(k_1,k_2)=(2,0)$ then after choosing the paths
     $(p_2,p_3,\dots,p_{i+1})$ and $(p_{i+4},\dots,p_{k-1})$
    the antipodal pair $p_1$ of $p_3$ and $p_{i+2}$ of $p_{i}$ is uniquely determined. Further, there are at most $2$ choices for $p_k$ and $p_{i+3}$. Thus, the total number of cycles of this type, using Theorem \ref{thm path}, is bounded by
    \begin{multline*}P^S_{k_1+6}(n)\cdot P^S_{k_2}(n)=\tilde{O}(n^{\lfloor2(k_1+9)/5 \rfloor+\lfloor 2(k_2+3)/5 \rfloor-2/3})= \tilde{O}(n^{\frac{2k+1}{5}-\frac{2}{3}})
    =\tilde{O}(n^{\lfloor 2k/5 \rfloor+1/3}).
    \end{multline*}
    \item If $(k_1,k_2)=(3,4)$ then after choosing the paths
     $(p_2,p_3,\dots,p_{i+1})$ and $(p_{i+4},\dots,p_{k-1})$
    the antipodal pair $p_1$ of $p_3$ and $p_{i+2}$ of $p_{i}$ is uniquely determined. Further, there are at most $2$ choices for $p_k$ and $p_{i+3}$. Thus, the total number of cycles of this type, using Theorem \ref{thm path}, is bounded by
    \begin{multline*}P^S_{k_1+6}(n)\cdot P^S_{k_2}(n)=\tilde{O}(n^{\lfloor2(k_1+9)/5 \rfloor+\lfloor 2(k_2+3)/5 \rfloor})= \tilde{O}(n^{(2k-4)/5})
    =\tilde{O}(n^{\lfloor 2k/5 \rfloor+1/3}).
    \end{multline*}
\end{itemize}

We also have to consider the more "degenerate" cases, when the paths $(p_k,p_1,p_2,p_3,p_4)$ and $(p_{i-1},p_i,p_{i+1},p_{i+2},p_{i+3})$ are either not disjoint or their complement consists only of one path. That is, up to symmetry we have to consider the cases when $i\in \{2,4,5,6\}$. We obtain the bound $\tilde{O}(n^{\lfloor 2(k+3)/5 \rfloor +1/3})$ 
\begin{itemize}
    \item for $i=2$ by choosing the $(k-3)$-path $(p_3,p_4,\dots,p_{k-1})$ first,
    \item for $i=4$ by choosing the $4$-path $(p_2,p_3,p_4,p_5)$ and $(k-8)$-path $(p_8,\dots,p_{k-1})$ first,
    \item for $i=5$ by choosing the $(k-3)$-path $(p_6,p_7,\dots,p_k,p_1,p_2)$ first.
    \item and for $i=6$ by choosing the $6$-path $(p_k,p_1,\dots,p_5)$ and $(k-10)$-path $(p_8,p_9\dots,p_{k-3})$ first. 
\end{itemize}

\noindent $\boldsymbol{k = 4 \mod 5: }$ 
Assume that there are two antipodal pairs $(p_1,p_{3})$ and $(p_i,p_{i+2})$ such that \mbox{$i\notin \{1,3\}$.} We will assume that the $5$-paths $(p_k,p_1,p_2,p_3,p_4)$ and $(p_{i-1},p_i,p_{i+1},p_{i+2},p_{i+3})$ are disjoint, and their complement consists of two non-empty paths $(p_5,\dots,p_{i-2})$ and $(p_{i+4},\dots,p_{k-1})$ of lengths $k_1$ and $k_2$ respectively. The "degenerate" cases can be done similarly as in the $k= 2 \mod 5$ case, thus we omit the details. Since $k_1+k_2= 4 \mod 5$, we may assume without loss of generality that $(k_1,k_2)=(0,4), (1,3)$ or $(2,2) \mod 5$. The analysis for the $(0,4)$ and $(1,3)$ cases are similar to the argument in the $k= 2$ case.

\begin{itemize}
    \item If $(k_1,k_2)=(1,3)$ then after choosing $(p_2,p_3,\dots,p_{i-2})$ and $(p_{i+1},p_{i+2},\dots,p_{k-1})$ arbitrarily, the antipodal pair $p_1$ of $p_3$ and $p_{i}$ of $p_{i+2}$ is uniquely determined. Further, there are at most $2$ choices for $p_k$ and $p_{i-1}$. Thus, the total number of cycles of this type, using Theorem \ref{thm path}, is bounded by
    \begin{multline*}P^S_{k_1+3}(n)\cdot P^S_{k_2+3}(n)=\tilde{O}(n^{\lfloor2(k_1+6)/5 \rfloor+\lfloor 2(k_2+6)/5 \rfloor}) \\
    =\tilde{O}(n^{(2k_1+9)/5+(2k_2+8)/5})=\tilde{O}(n^{(2k-3)/5})=\tilde{O}(n^{\lfloor 2k/5 \rfloor}).
    \end{multline*}
    \item If $(k_1,k_2)=(0,4)$ then after choosing the paths
     $(p_2,p_3,\dots,p_{i+1})$ and $(p_{i+4},\dots,p_{k-1})$
    the antipodal pair $p_1$ of $p_3$ and $p_{i+2}$ of $p_{i}$ is uniquely determined. Further, there are at most $2$ choices for $p_k$ and $p_{i+3}$. Thus, the total number of cycles of this type, using Theorem \ref{thm path}, is bounded by
    \begin{multline*}P^S_{k_1+6}(n)\cdot P^S_{k_2}(n)=\tilde{O}(n^{\lfloor2(k_1+9)/5 \rfloor+\lfloor 2(k_2+3)/5 \rfloor})
    =n^{(2k_1+15)/5+(2k_2+2)/5}=\tilde{O}(n^{\lfloor 2k/5 \rfloor}).
    \end{multline*}
    \item If $(k_1,k_2)=(2,2)$ an argument similar to the one used for the $(k_1,k_2)=(1,3)$ and $(0,4)$ cases is not sufficient, and we need to consider a few other cases depending whether there are other antipodal pairs. (The difficulty of this case comes from the following fact. When we only break up the cycle into two paths, there lengths are in a way that the product of the number of paths of the corresponding length is too large). 
    Assume first that there are no other antipodal pairs. After choosing $(p_2,p_3,\dots,p_{i+1})$ and $(p_{i+4},\dots,p_{k-1})$, the antipodal pair $p_1$ of $p_3$ and $p_{i+2}$ of $p_i$ is uniquely determined. Further, there are at most $2$ choices for $p_k$ and $p_{i+1}$. We bound the number of paths $(p_2,p_3,\dots,p_{i+1})$ and $(p_{i+4},\dots,p_{k-1})$ by Proposition \ref{prop planar}, and obtain the bound
    \[4P_{k_1+2}P_{k_2}=\tilde{O}(n^{\lfloor 2k/5\rfloor +\frac{1}{3}}).\]
    Indeed, for $k\geq 19$ it follows from 
    \[4P_{k_1+2}P_{k_2}= \tilde{O}(n^{(k_1+6)/3+1}n^{k_2/3+1})=\tilde{O}(n^{(k-3)/3+2})\leq \tilde{O}(n^{\lfloor2k/5\rfloor}),\]
    and for $k\leq 15$ it follows by using the exact bounds from Proposition \ref{prop planar}.
    
    Thus, we may assume that without loss of generality there is a third antipodal pair $(p_j,p_{j+1})$ with $2\leq j \leq i-1$. Again, we assume that the $5$-paths $(p_k,p_1,p_2,p_3,p_4)$, $(p_{j-1},p_j,p_{j+1},p_{j+2},p_{j+3})$ and $(p_{i-1},p_i,p_{i+1},p_{i+2},p_{i+3})$ are pairwise disjoint, and their complement consists of non-empty paths $(p_5,\dots,p_{j-2})$, $(p_{j+4},\dots,p_{i-2})$ and $(p_{i+4},\dots,p_{k-1})$, of lengths $\ell_1,\ell_2$ and $k_2$ respectively. Without loss of generality we may assume that $(\ell_1,\ell_2,k_2)=(0,2,2), (1,1,2)$ or $(3,4,2)$. In any of these cases, using the antipodal pairs $(p_1,p_3)$ and $(p_j,p_{j+1})$, we can proceed as in the $(k_1,k_2)=(0,4)$ or $(1,3)$ cases. 
    
    Finally, we note that the degenerate cases, when two antipodal pairs are too close to each other, can be done similarly as in the $k= 2 \mod 5$ case.
\end{itemize}
\end{proof}

The next proposition summarizes the bounds for short cycles.

\begin{proposition}\label{prop small k}
We have
\begin{alignat*}{4}
C_3(n) & = \Omega(n) & \textrm{ and } & \textrm {  }&  \textrm {  }& C_3(n)& = O(n^{4/3}), \\
C_6(n)& = \Omega(n^2 \log \log n) & \textrm { and } & \textrm {  } &  \textrm {  } & C_6(n)& = \tilde{O}(n^{20/9}), \\
C_7(n)& =\Omega(n^{7/3}) & \textrm{ and } & \textrm {  } &  \textrm {  }&  C_7(n)& = \tilde{O}(n^{8/3}),\\
C_9(n)& =\Omega(n^{3}) & \textrm{ and } & \textrm {  } &  \textrm {  } & C_9(n)& = \tilde{O}(n^{10/3}).
\end{alignat*}
\end{proposition}

It would be interesting to find sharp bounds for short cycles too. The lower bound $C_6(n)=\Omega(n^2\log \log n)$, via stereographic projection, is an immediate corollary of a construction by Klávik, Král and Mach \cite{Kral} giving $\Omega(n^2\log n \log n)$ $6$-cycles in an incidence graph of $n$ points and $n$ lines. This disproved a conjecture of de Caen and Székely \cite{Szekely} that the maximum number of $6$-cycles is $O(n^2)$ in the point-line incidence graph. It is not hard to see that the number of those $6$-cycles that contain an antipodal pair in the unit distance graph on the sphere is $O(n^2)$. Thus, up to the order of magnitude, the problem of bounding the number of $6$-cycles in point line incidence graphs and in unit distance graphs on the sphere, are equivalent.

There is a similar explanation why the $C_7$ and $C_9$ cases are more difficult than the longer cycles: It is not hard to prove upper bounds matching the lower bounds for the number of those cycles, in which there are is least one antipodal pair. Thus, again, for $C_7$ and $C_9$ the most difficult types of cycles to bound are those in which there are no antipodal pairs. For longer cycles, there is no similar issue. This is because no antipodal pair means we can use the bound from Proposition \ref{prop planar} for long sub-paths, and even with a wasteful estimate we will obtain sufficiently good bounds.

\begin{proof}[Proof of Proposition \ref{prop small k}] We start with the lower bounds. For $k=3$ finding linear lower bounds is trivial, and for $k=6$ it follows from \cite{Kral} combined with stereographic projection. For $k=7,9$ we can use the same constructions as for the $k\geq 10$ case. We now turn to the lower bound.

\noindent $\boldsymbol{k=3:}$ The upper bound follows from the Szemerédi-Trotter bound, since every edge can be extended in at most two different ways to a triangle.

For the other cases we use a nested dyadic decomposition argument, similar to the one used by \cite{almostsharp}. The proof for the $k=6,7,9$ cases are all similar, but the $k=6$ cases is somewhat harder. Thus we only spell out the proof for the $k=6$ case and omit the details for $k=7,9$.

Recall that by \eqref{rich} that for any $0\leq \alpha \leq 1$ the maximum number of $n^{\alpha}$-rich points is \mbox{$O(\max \{n^{3-2\alpha},n^{1-\alpha}\})$.} We call a point \emph{usual} if it is $n^{\alpha}$-rich for some $0\leq \alpha \leq 1/2$ and \emph{very rich} if it is $n^{\alpha}$-rich for some $\frac{1}{2}\leq \alpha \leq 1$.

\noindent $\boldsymbol{k=6:}$ Bounding the number of those $6$-cycles in which there are antipodal pairs, can be done similarly as for the $k\geq 10$ cases. Thus, we may assume that there are no antipodal pairs.
First, we bound the number of those $6$-cycles in which there are no four consecutive usual points. In any such cycle, we can find two disjoint edges, separated by another vertex in each direction, that both have a very rich endpoint. Assume that they are $n^{\alpha_1}$ and $n^{\alpha_2}$-rich respectively. Since there are no antipodal pairs in the cycle, after choosing the two disjoint edges, there are at most $4$ different ways to extend it to a $6$-cycle. Thus, by using dyadic decomposition and \eqref{rich}, we obtain the bound
\[\sum_{(\alpha_1,\alpha_2\in \Lambda)}O(n^{1-\alpha_1}n^{\alpha_1}n^{1-\alpha_2}n^{\alpha_2})=\tilde{O}(n)^2,\]
where $\Lambda=\setbuilder{(i,j)}{i,j\in \{\lfloor \frac{\log_2 n}{2}\rfloor,\lfloor \frac{\log_2 n}{2}\rfloor+1,\dots, \lfloor\log_2 n \rfloor\}}$.

Next, we bound the number of those $6$-cycles $(p_1,p_2,p_3,p_4,p_5,p_6)$, in which there are four consecutive usual points, say $(p_2,p_3,p_4,p_5)$.

For some $0\leq \alpha_2,\alpha_3,\alpha_4,\alpha_5 \leq \frac{1}{2}$ let $Q_2(\alpha_2)$ be the set of those points that are at least $n^{\alpha_2}$-rich and at most $2n^{\alpha_2}$-rich, and let $Q_5(\alpha_5)$ be the set of those points that are at least $n^{\alpha_5}$-rich and at most $2n^{\alpha_5}$-rich. Further, let $Q_3(\alpha_2,\alpha_3)$ be the set of those points that are at least $n^{\alpha_3}$-rich and at most $2n^{\alpha_3}$ with respect to $Q_2(\alpha_3)$, and let $Q_4(\alpha_4)$ be the set of those points that are at least $n^{\alpha_4}$-rich and at most $2n^{\alpha_4}$-rich with respect to $Q_5(\alpha_5)$.
Finally, let $Q_1(\alpha_2,\alpha_3)$ be the union of the second neighborhoods of points in $Q_3(\alpha_3)$, and $Q_6(\alpha_4,\alpha_5)$ be the union of the second neighborhoods of the points in $Q_4(\alpha_4)$.

Using dyadic decomposition, it is sufficient to show that for any fixed $0\leq \alpha_2,\alpha_3,\alpha_4,\alpha_5 \leq \frac{1}{2}$ the number of those $6$-cycles $(p_1,\dots,p_6)$ such that $p_i\in Q_i(\alpha_i)$ for $i\in \{2,3,4,5\}$ is $O(n^{20/9})$. Note that for any such cycle $p_1$ must be in $Q_1(\alpha_2,\alpha_3)$ and $p_6$ must be in $Q_6(\alpha_4,\alpha_5)$. In the rest of the proof we will use the notation $Q_i=Q_i(\alpha_i)$ for $i\in \{2,3,4,5\}$, and $Q_6=Q_6(\alpha_4,\alpha_5)$, $Q_1=Q_1(\alpha_2,\alpha_3)$.

Let $0\leq x_1,\dots,x_6 \leq 1$ such that $|Q_i|=n^{x_i}$. First, we bound the number of $6$-cycles in the case when at least two $Q_i$ is of size $O(n^{1/2})$. If there are two cyclically adjacent indices, say $1$ and $2$, such that $Q_1$ and $Q_2$ are of size $O(n^{1/2})$, then we obtain the bound \[4u(n^{x_1},n^{x_6})u(n^{x_3},n^{x_4})=O(n^{2/3}n^{4/3})=O(n^2),\] by picking $(p_1,p_6)$, $(p_3,p_4)$ and extending $(p_2,p_3,p_4,p_6)$ to a $6$ cycle in at most $4$ different ways. If there are two such non-adjacent $i$ and $j$ such that $Q_i$ and $Q_j$ are of size $O(n^{1/2})$, then we will find two disjoint pairs of indices, separated by $1$ index in each direction, say $(1,6)$ and $(3,4)$ such that $u(n^{x_1},n^{x_6})=O(n)$ and $u(n^{x_3},n^{x_4})=O(n)$, and obtain the bound $O(n^2)$ again.

Next, we bound the number of $6$-cycles in the case when at least one $Q_i$, say $Q_1$ is of size $O(n^{2/9})$. In this case, by picking $(p_3,p_4,p_5)$ and $p_1$ first, we can extend it to a $6$-cycle in at most $4$ different ways, and obtain the bound $O(n^2)n^{2/9}=O(n^{20/9})$. 

From now on, we assume that there are at most one $Q_i$ with $|Q_i|=O (n^{\frac{1}{2}})$, and every $Q_i$ is of size $\Omega(n^{2/9})$. We count the $6$-cycles by picking $(p_3,p_4)$ first. Then $p_1$ must be in the second neighborhood of $p_3$, and $p_6$ must be in the second neighborhood of $p_4$. Further, once $(p_1,p_3,p_4,p_6)$ is picked, there are at most $4$ different ways to finish the cycles. With this, we obtain the bound
\begin{equation}\label{eq c6}
u(n^{x_3},n^{x_4})u(4n^{\alpha_2+\alpha_3},4n^{\alpha_4+\alpha_5}).
\end{equation}
We also have
\begin{equation}\label{eq alpha}
n^{\alpha_2}\leq \frac{u(n^{x_2},n)}{n^{x_2}},\textrm{ } \textrm{ } n^{\alpha_3}\leq \frac{u(n^{x_2},n^{x_3})}{n^{x_3}}, \textrm{ } \textrm{ } n^{\alpha_4}\leq \frac{u(n^{x_4},n^{x_5})}{n^{x_4}}, \textrm{ } \textrm{ } n^{\alpha_5}\leq \frac{u(n^{x_5},n)}{n^{x_5}}.
\end{equation}

By \eqref{SzT} we have $u(m,n)=O(m^{2/3}n^{2/3}+n+m)=O(\max\{m^{2/3}n^{2/3},n,m\})$. We will distinguish a few cases based on which term the maximum is taken in \eqref{eq c6} and in the inequalities in \eqref{eq alpha}.

\textbf{Case 1: } Both in \eqref{eq c6} and \eqref{eq alpha} the maximum is taken on the first term everywhere. Then we obtain the bound
\begin{equation*}
u(n^{x_3},n^{x_4})u(4n^{\alpha_2+\alpha_3},4n^{\alpha_4+\alpha_5})= O\left(n^{\frac{2}{3}(x_3+x_4)}n^{\frac{2}{3}(\frac{4}{3}+\frac{1}{3}(x_2+x_5)-\frac{1}{3}(x_3+x_4))}\right )=O(n^{20/9}).
\end{equation*}

\textbf{Case 2: } $u(n^{x_3},n^{x_4})=O(\max\{n^{x_3},n^{x_4}\})$. Without loss of generality we may assume that $u(n^{x_3},n^{x_4})=O(n^{x_4})$. Note that this implies $n^{x_3}=O(n^{1/2})$. Then we may assume that $u(4n^{\alpha_2+\alpha_3},4n^{\alpha_4+\alpha_5})=O(n^{\frac{2}{3}(\alpha_2+\alpha_3+\alpha_4+\alpha_5)})$, otherwise we would obtain the bound $O(n^2)$. Similarly, we may assume that the maximum in the bound for $u(n^{x_2},n), u(n^{x_4,x_5}),u(n^{x_5},n)$ the maximum is taken on the first term, otherwise we would obtain two parts $Q_i$ of size $O(n^{1/2})$.

\begin{itemize}
\item If $u(n^{x_2},n^{x_3})=O(n^{x_3})$, then $n^{x_2}=O(n^{1/2})$, giving a $Q_i$ of size $O(n^{1/2})$. 
\item If $u(n^{x_2},n^{x_3})=O(n^{x_2})$, then we obtain the bound 
\begin{equation*}
u(n^{x_3},n^{x_4})u(4n^{\alpha_2+\alpha_3},4n^{\alpha_4+\alpha_5})= O\left(n^{x_4}n^{\frac{2}{3}(\frac{4}{3}+\frac{2}{3}x_2-x_3-\frac{1}{3}x_4+\frac{1}{3}x_5)}\right )=O(n^{20/9}),  
\end{equation*}
using the assumption that $n^{x_3}=\Omega(n^{2/9})$. 
\item Finally, if $u(n^{x_2},n^{x_3})=u(n^{\frac{2}{3}(x_2+x_3)})$, then we obtain
\[u(n^{x_3},n^{x_4})u(4n^{\alpha_2+\alpha_3},4n^{\alpha_4+\alpha_5})=O\left(n^{x_4}n^{\frac{2}{3}(\frac{4}{3}+\frac{1}{3}(x_2+x_5)-\frac{1}{3}(x_3+x_4))}\right )=O(n^{20/9}).\]
\end{itemize}

\textbf{Case 3: } $u(4n^{\alpha_2+\alpha_3},4n^{\alpha_4+\alpha_5})=O(\max\{n^{\alpha_2+\alpha_3},n^{\alpha_4+\alpha_5}\})$. Without loss of generality we may assume that $u(4n^{\alpha_2+\alpha_3},4n^{\alpha_4+\alpha_5})=O(\max\{n^{\alpha_2+\alpha_3}\})$. This implies $n^{\alpha_4+\alpha_5}=O(n^{1/2})$. Similarly as in Case 2, we may assume that in the bound for $u(n^{x_2},n)$, $u(n^{x_2},n^{x_3})$, $u(n^{x_3},n^{x_4})$ the maximum is taken in the first term. Then regardless on which term the maximum is taken in $u(n^{x_5},n)$, we obtain the bound

\[u(n^{x_3},n^{x_4})u(4n^{\alpha_2+\alpha_3},4n^{\alpha_4+\alpha_5})=O\left(n^{\frac{2}{3}(x_3+x_4)}n^{\frac{2}{3}+\frac{1}{3}x_2-\frac{1}{3}x_3}\right )=O(n^{20/9}).\]

$\boldsymbol{k=7,9:}$ The proof is by using dyadic decomposition in a similar way as in the $k=6$ case. The reason why there is $3$ in the denominator of the exponent instead of $9$ is that we do not consider $Q_1(\alpha_2,\alpha_3)$ type of sets, only $Q_i(\alpha_i)$.

\bigskip

\end{proof}

\section{3-Regular graphs in $\mathbb{R}^3$}

The main goal of this section is to prove Theorem \ref{thm 3-reg}. As it does not affect the answer up to the order of magnitude, we switch to the multipartite version of the problem. For a fixed $3$-regular graph $G$ on $k$ vertices, and for $k$ sets $P_1,\dots,P_k\subseteq \mathbb{R}^3$ we denote by $F(P_1,\dots,P_k)$ the maximum number of $k$-tuples $(p_1,\dots,p_k)$ such that the unit distance graph determined by them is isometric to $G$, and $p_i\in P_i$ for every $i\in [k]$. Further, we use the notation
\[f(n_1,\dots,n_k)=\max|F(P_1,\dots,P_k)|,
\]
and $f(n)=f(n,\dots,n)$.

To prove Theorem \ref{thm 3-reg}, we follow a divide and conquer strategy of Agarwal and Sharir \cite{cuttinglemma}. The strategy uses \emph{cuttings}, a partitioning technique, which was a precursor to the more recent polynomial partitioning method. We say that a sphere $S$ \emph{crosses} a subset $\tau\subseteq \mathbb{R}^n$ if $S \cap \tau \neq \emptyset$, but $\tau \not\subset S$. To follow usual terminology, we will call the subsets in the partition \emph{cells}. (Note that usually a \emph{cell} in this context, means a more specific subset described by a bounded number of polynomials. However, since we only use the cutting results as a black-box in a very specific case, for simplicity, we do not define cells here more properly.)

The following cutting lemma was proved in \cite{cuttinglemma}).

\begin{lemma}[Cutting Lemma from \cite{cuttinglemma}]\label{cutting lemma}
Given a set of points $P$ and $\ell$ sets $Q_0,\dots,Q_{\ell}$ of spheres in $\mathbb{R}^d$, for any $1\leq r \leq n$ we can partition the $\mathbb{R}^d$ into cells such that the following three conditions hold.
\begin{enumerate}
    \item The number of cells is $\tilde{O}(r^d)$.
    \item The number of points in each cell is $O\left (\frac{|P|}{r^d}\right )$.
    \item For every $i\in [\ell]$ each cell is crossed by $O\left (\frac{|Q_i|}{r}\right )$ many spheres from $Q_i$.
\end{enumerate}
Further, if $P$ is contained in a $(d-1)$-sphere $\mathbb{S}^{d-1}$, then we can partition $\mathbb{S}^{d-1}$ into cells such that the following three conditions holds.
\begin{enumerate}
    \item The number of cells is $\tilde{O}(r^{d-1})$.
    \item The number of points in each cell is $O\left (\frac{|P|}{r^{d-1}}\right)$.
    \item For every $i\in [\ell]$ each cell is crossed by $O\left( \frac{|Q_i|}{r}\right)$ many spheres from $Q_i$.
\end{enumerate}
\end{lemma}

To illustrate the method, we first prove the following simpler proposition, which gives the best upper bound we have been able to prove on the number of $4$-cycles in $\mathbb{R}^3$.

\begin{proposition}\label{prop: 4cycle}
The maximum number of unit $4$-cycles determined by a set of $n$ points in $\mathbb{R}^3$ is $\tilde{O}(n^{12/5})$.
\end{proposition}

It would be interesting to find better estimates. The best lower bounds we found is $\Omega(n^2)$, given by the same construction used on the sphere.

\begin{problem} Find the maximum possible number of $4$-cycles determined by a set of $n$ points in $\mathbb{R}^3$.
\end{problem}

\begin{proof}[Proof of Proposition \ref{prop: 4cycle}]
We again switch to the multipartite version of the problem, and for sets $P_1,P_2,P_3,P_4$  we denote by $C_4(P_1,P_2,P_3,P_4)$ the maximum number of unit $4$-cycles $(p_1,p_2,p_3,p_4)$ with $p_i\in P_i$. Further, let $c_4(n_1,n_2,n_3,n_4)=\max|C_4(P_1,P_2,P_3,P_4)|$ where the maximum is taken over all $P_i$ with $|P_i|\leq n_i$, and let $c_4(n)=c_4(n,n,n,n)$. For a point $p$ we also denote by $S(p)$ the unit sphere centered at $p$.
We will bound $C_4(P_1,P_2,P_3,P_4)$ with $|P_i|\leq n$ for all $i\in [4]$.

For some parameter $r$ we partition $\mathbb{R}^3$ into $\tilde{O}(r^3)$ cells as in Lemma \ref{cutting lemma} with $P_1$ as the set of points, and $\setbuilder{S(p_2)}{p_2\in P_2}$ and $\setbuilder{S(p_4)}{p_4\in P_4}$ as the sets of spheres. For a cell $\tau$ let $P_1^{\tau}=P_1\cap \tau$, further, for $i=2,4$ let $Q_i^{\tau}\subseteq \setbuilder{S(p_i)}{p_i\in P_i}$ be the set of those spheres that cross $\tau$, and $R_i^{\tau}\subseteq \setbuilder{S(p_i)}{p_i\in P_i}$ be the set of those spheres that contain $\tau$. Summing over all cells $\tau$ we obtain
\[C_4(P_1,P_2,P_3,P_4)\leq \sum_{\tau} \Big ( C_4(P_1^{\tau},Q_2^{\tau},P_3,Q_4^{\tau})+C_4(P_1^{\tau},R_2^{\tau},P_3,P_4)+C_4(P_1^{\tau},P_2,P_3,R_4^{\tau})\Big).
\]

\begin{proposition}\label{prop contain}$C_4(P_1^{\tau},R_2^{\tau},P_3,P_4)=\tilde{O}(n^2)$ and  $C_4(P_1^{\tau},P_2,P_3,R_4^{\tau})=\tilde{O}(n^2)$.
\end{proposition}
\begin{proof}
We will only prove $C_4(P_1^{\tau},R_2^{\tau},P_3,P_4)=\tilde{O}(n^2)$, which is sufficient by symmetry.
This means we have to bound the number of $4$-cycles under the condition that all points in $P_1^{\tau}$ are contained in the intersection of unit spheres centered around the points of $R_2^{\tau}$. Since in $\mathbb{R}^3$ for any $3$ points there are at most $2$ other points unit distance apart from each, we either have $|P_1^{\tau}|\leq 2$ or $|R_2^{\tau}|\leq 2$. Thus, $C_4(P_1^{\tau},R_2^{\tau},P_3,P_4)$ is bounded by four times the maximum number of $3$-paths in $\mathbb{R}^3$, which is $\tilde{O}(n^2)$ by \cite{almostsharp}. 
\end{proof}

By the properties of the cutting we also have $C_4(P_1^{\tau},Q_2^{\tau},P_3,Q_4^{\tau})\leq c_4(\frac{n}{r^3},\frac{n}{r},n,\frac{n}{r})$. Thus, we obtain
\[C_4(P_1,P_2,P_3,P_4)\leq \sum_{\tau} \Big ( c_4\left (\frac{n}{r^3},\frac{n}{r},n,\frac{n}{r}\right)+\tilde{O}(n^2)\Big).
\]

Repeating a similar analysis three more times with cyclic shifts (in the next round $P_2$ plays the role of $P_1$, $P_3$ plays the role of $P_2$, $P_4$ play the role of $P_3$, and $P_1$ plays the role of $P_4$, and so on), and using that the number of cells in the cutting is $\tilde{O}(r^3)$, we obtain the recurrence
\[c_4(n) \leq \tilde{O}(r^{12})c_4\left (\frac{n}{r^5}\right ).\]
With an appropriate choice of $r$ this recursion yields $c_4(n)\leq \tilde{O}(n^{12/5})$.
\end{proof}

Notice that in the proof of Proposition \ref{prop: 4cycle} in each round when the Cutting Lemma is applied, for every cell we have to consider two different situations and split into two subproblems. For those spheres that cross the cell, we directly plug in the bound on the number of crossings from the Lemma, and obtain the main term of the recursion. Those spheres that contain the cell are  accounted for in Proposition \ref{prop contain}. Proving Proposition \ref{prop contain} was simple, and we could deal with it by `hand`. However, when we work with a large $3$-regular graph, knowing information only about one local containment situation does not make the problem sufficiently simpler.

Therefore, we will follow further ideas of Agarwal and Sharir that they developed for bounding the number of $k$-simplices in higher dimension. We will sketch these ideas for completeness with incorporating the sufficient changes in the method, adjusting it to our problem. Notice that for $d=0,1$ Lemma \ref{cutting lemma} is trivial. Yet, we will utilize these trivial cases, as they will give a convenient uniform way to handle the containment situations mentioned in the previous paragraph.

\bigskip

\begin{proof}[Proof of Theorem \ref{thm 3-reg}]
We will use the Cutting Lemma in $k$-rounds to derive a recurrence for $f(n)$, following Section $5$ of \cite{cuttinglemma}, with making some suitable changes. We will assume that the vertex set of $G$ is $[k]$.

Let $P_1,\dots, P_k\subseteq \mathbb{R}^3$ be point sets with $|P_i|\leq n_i$ for every $i\in [k]$. We denote by $S(p)$ the unit sphere centered at $p$. Without loss of generality, we may assume that $2,3,4$ are the neighbors of $1$.

In the first round, we use the Cutting Lemma with some parameter $r$, with $P_1$ as the set of points, and $Q_i=\setbuilder{S(p)}{p\in P_i}$ as the sets of spheres for $i=2,3,4$. For a cell $\tau$ let $P_1^{\tau}=P_1\cap \tau$. Further, for $i=2,3,4$ let $Q_i^{\tau}\subseteq Q_i$ be the set of those spheres that cross $\tau$, and $R_i^{\tau}$ be the set of those spheres that contain $\tau$.
Summing over all cells $\tau$ we obtain
\begin{multline}\label{ 1st cutting}F(P_1,\dots,P_k)=\sum_{\tau}\Big (F(P_1^{\tau},Q_2^{\tau},Q_3^{\tau},Q_4^{\tau},P_5,\dots,P_k)+F(P_1^{\tau},R_2^{\tau},P_3,\dots,P_k)\\
+F(P_1^{\tau},P_2,R_3^{\tau},P_4,\dots,P_k) +F(P_1^{\tau},P_2,P_3,R_4^{\tau},P_5\dots,P_k)\Big).
\end{multline}

We will bound the first term in each summand by plugging in the information from the Cutting Lemma. In the proof of Proposition \ref{prop: 4cycle} we bounded the terms similar to the other three terms in Proposition \ref{prop contain}, by observing that they correspond to bounding cycles in a geometrically constrained situation. While it is still true here that for example either $|P_1^{\tau}|\leq 2$ or $|R_2^{\tau}|\leq 2$, for large $k$ it does not constrain the geometry sufficiently to bound $F(P_1,R_2^{\tau},P_3,\dots,P_k)$ easily.

Thus, we will introduce new sub-problems, where we will keep track of containments that occur between certain parts by a weighted auxiliary graph $H$. Notice that we can say more than just $|P_1^{\tau}|\leq 2$ or $|R_2^{\tau}|\leq 2$. It is also true that if $|P_1^\tau|\geq 3$, then $P_1$ is contained in a $2$-sphere, and a similar observation holds for $R_2$.

For a subgraph $H$ of $G$ we say that a $k$-tuple of points $(P_1,\dots,P_k)$ is of \emph{type $H$}, if  for every edge $(i,j)$ of $H$ the distance between every point of $P_i$ and $P_j$ is the one.
We denote by 
\[F^{H}(n_1,\dots,n_k)=\max F(P_1,\dots,P_k),\]
where the maximum is taken over all $k$-tuples $(P_1,\dots,P_k)$ of type $H$ with $|P_i|\leq n_i$.

With this notation, \eqref{ 1st cutting} implies
\begin{multline*}\label{ 1st cutting}F(n_1,\dots,n_k)= \\[3pt] \tilde{O}(r^3)\left (F\left (\frac{n_1}{r^3},\frac{n_2}{r},\frac{n_3}{r},\frac{n_4}{r},n_5,\dots,n_k\right )+F^{H_2}(n_1,\dots,n_k)+F^{H_3}(n_1,\dots,n_k)+F^{H_4}(n_1,\dots,n_k)\right),
\end{multline*}
where $H_i$ is the graph with a single edge $(1,i)$.

Refining the notion further, for a vector $\pmb{\lambda}=(\lambda_1,\dots,\lambda_k)\in \{0,1,2,3\}^k$ we say that a $k$-tuple $(P_1,\dots,P_k)$ is of type $(H,\pmb{\lambda})$,
\begin{itemize}
    \item if $(P_1,\dots,P_k)$ is of type $H$, and
    \item if $\lambda_i\leq 2$, then $P_i$ is contained in a $\lambda_i$-sphere but is not contained in a $\lambda_i-1$-sphere, further
    \item if $\lambda_i=3$, then $P_i$ is not contained in a sphere.
\end{itemize}
We define $F^{H,\pmb{\lambda}}(n_1,\dots,n_k)$ analogously to $F^H(n_1,\dots,n_k)$. Further, we say that a pair $(H,\pmb{\lambda})$ is \emph{realizable}, if there exists a $k$-tuple $(P_1,\dots,P_k)$ of type $(H,\pmb{\lambda})$.

For any subgraph $H$ of $G$, and any vector $\pmb{\lambda}$ by applying the cutting lemma with the $i$-th part playing the role of the points, and with some parameter $r_i$, we obtain
\begin{multline*}F^{H,\pmb{\lambda}}(n_1,\dots, n_k)=\\[3pt] \tilde{O}(r^3)\Big ( F^{H,\pmb{\lambda}}(m_1,\dots,m_k)+ F^{H_1,\pmb{\lambda}_1}(n_1,\dots,n_k)+F^{H_2,\pmb{\lambda}_2}(n_1,\dots,n_k)+F^{H_3,\pmb{\lambda}_3}(n_1,\dots,n_k) \Big),
\end{multline*}
where
\begin{itemize}
    \item $m_i=\frac{n^i}{r^{\lambda_i}}$,
    \item $m_j=\frac{n_j}{r}$ if $(i,j)\in G\setminus H$,
    \item $m_j=n_j$ otherwise,
    \item each $H_i$ is $H$ extended by an edge connecting $i$ with one of its neighbours in $G$, and
    \item $\pmb{\lambda}_i$ is $\pmb{\lambda}$ modified suitably along the new edge.
\end{itemize}

By applying the cutting lemma similarly $k$ times such that in the $i$-th round we have the points of the $i$-th part as the set of points, the spheres centred in the neighbours (in $G$) of $i$ as spheres, and $r_i=r^{x_i}$ as parameter, for any subgraph $H$ we obtain
\begin{equation*}\label{eq: rec}
    F^{H,\pmb{\lambda}}(n_1,\dots,n_k)=\tilde{O}\Big (\prod_i{r^{x_i}}\Big) \Big(F^{H,\pmb{\lambda}}(m_1,\dots,m_1)+\sum_{H'}F^{H'}(n_1,\dots,n_k)\Big)
\end{equation*}
where the sum is taken over all subgraphs $H'$ of $G$ with strictly more edges than $H$, and where 
\[m_i=\frac{n_i}{r^{x_i\lambda_i+\sum_{(i,j)\in G\setminus H}x_j}}.\]

Let $\xi(H,\pmb{\lambda})$ be the solution of the following linear optimization problem:
\begin{align}\label{total}
&\min \sum {x_i\lambda_i} \ \ \ \text{subject to: }\\
&x_i\geq 0 \ \ \ \ \ \ \ \ \ \ \ \ \ \, \ \ \    \textrm { for } 1\leq i \leq k\\
\label{cond}
&\lambda_ix_i+\sum_{(i,j)\in G\setminus H} x_j\geq 1 \ \  \textrm{ for } 1\leq i \leq k
\end{align}

With this, we obtain
\begin{equation}\label{eq: rec2}
    F^{H,\pmb{\lambda}}(n)=\tilde{O}\Big (r^{\xi(H,\pmb{\lambda})}\Big) \Big(F^{H,\pmb{\lambda}}\Big(\frac{n}{r}\Big)+\sum_{H'}F^{H'}(n)\Big ).
\end{equation}

Let $\xi=\max_{H,\pmb{\lambda}}{\xi(H,\pmb{\lambda})}$, where the maximum is taken over all realizable $(H,\pmb{\lambda})$. By induction on the number of edges in $G\setminus H$, and with an appropriate choice of $r$, using \eqref{eq: rec2} one can show that we have $F(n)=\tilde{O}({n^{\xi}})$. Note that the starting case of the induction is $G=H$, for which one can show directly that $F^G(n)=O(n^{\frac{k}{2}})$. Indeed, if $(P_1,\dots,P_k)$ is of type $G$, then $|P_i|\leq 2$ for at least $\frac{k}{2}$ indices $i$. 

Thus, it is sufficient to show that the solution of the linear optimization problem for any realizable pair $(H,\pmb{\lambda})$ is at most $\frac{k}{2}$.

\bigskip

We make some geometric observations about realizable pairs $(H,\pmb{\lambda})$: 
\begin{enumerate}[label=(\roman*)]
    \item If $\lambda_i=3$, then $i$ is an isolated vertex of $H$. \label{lambda 3}
    \item If $\lambda_i=2$, then $i$ has exactly one neighbour $j$ in $H$, for which we must have $\lambda_j=0$. (As $P_j$ must be in the center of a $2$-sphere) \label{lambda 2}
    \item If $\lambda_i=1$, then for any neighbour $j$ of $i$ in $H$ we must have $\lambda_j=0$ (As there are at most two points at a given distance from all points of a circle.) \label{lambda 1}
    \item If $i$ is an isolated vertex in $H$, then we may assume that for every neighbour $j$ of $i$ in $G$ we have $\lambda_j\geq 1$. \label{isolated}
    \item We may assume that if $\lambda_i=0$, then $\deg_H(i)=3$. (Indeed, if $\lambda_i=0$, and $(P_1,\dots,P_n)$ is of type $(H,\pmb{\lambda})$, then $|P_i|\leq 2$. Since we only want to bound $F^{H,\pmb{\lambda}}(n)$ up to the order of magnitude, we may assume that $|P_i|=1$. Then if $(i,j)\in G$, we may discard those $p_j\in P_j$ that are not at distance one from the single point of $P_i$, as they cannot be part of any copy of $G$.) \label{lambda 0}
\end{enumerate}.

Let us define $\mathbf{x}=(x_1,\dots,x_k)$ as

\begin{equation*}
x_i=
\begin{cases*}
\frac{1}{2\lambda_i} & if $\lambda_i\neq  0$ and $\textrm{deg}_H(i)=0$ \\[4pt]
\frac{1-\frac{1}{6}(3-\deg_H(i))}{\lambda_i} & if  $\lambda_i\neq 0$ and $\textrm{deg}_H(i)\neq 0$\\[4pt]
 0 & if $\lambda_i=0$.
\end{cases*}
\end{equation*}

We will show that this $\mathbf{x}$ satisfies the constraints in \eqref{cond} in the linear optimisation problem, and gives optimal-function value $\frac{k}{2}$ in \eqref{total}. This will finish the proof.

To check \eqref{cond} we point out that each vertex $i$ gets a contribution from those vertices $j$ that are neighbours if $i$ in $G$  but not in $H$. Notice that $x_j\geq \frac{1}{6}$ if $\lambda_j\neq 0$. This, together with \ref{lambda 0} implies that
{\[\lambda_ix_i+\sum_{(i,j)\in G\setminus H} x_j\geq 1-\frac{1}{6}(3-\deg_H(i))+(3-\deg_H(i))\frac{1}{6}=1.\]}

To finish, we show that $\displaystyle{\Sigma_{i=1}^k} \lambda_i x_i \leq k/2$. Let
\begin{align*}
k_0 &= |\setbuilder{i}{\deg_H(i)=0}|, \\
k_1 &= |\setbuilder{i}{\deg_H(i)=1, \lambda_i=2}|, \\[2pt]
k_2 & = |\setbuilder{i}{\deg_H(i)=3, \lambda_i=1}|, \\[2pt]
k_3 &= |\setbuilder{i}{\deg_H(i)=1, \lambda_i=1}|,\\[2pt]
k_4 &= |\setbuilder{i}{\deg_H(i)=2, \lambda_i=1}|, \\
k_5 &= |\setbuilder{i}{\lambda_i=0}|.
\end{align*}

Then \ref{lambda 2}, \ref{lambda 1} and \ref{lambda 0} together with a double counting implies that
\begin{equation*}
   3k_5 =k_1+3k_2+k_3+ 2k_4,
\end{equation*}
which gives that
\begin{equation*}
    k=k_0+k_1+k_2+k_3+k_4+k_5=k_0+ \frac{4}{3}k_1+2k_2+\frac{4}{3}k_3+\frac{5}{3}k_4.
\end{equation*}

From this, and the definition of $\mathbf{x}$ it follows that
\begin{equation*}
    \sum_{i=1}^k\lambda_ix_i=\frac{1}{2}k_0+\frac{1}{3}k_1+k_2+\frac{2}{3}k_3+\frac{5}{6}k_4 \leq \frac{k}{2}.
\end{equation*}
\end{proof}

We close this section by describing constructions for bipartite $G$ that match the upper bound in the slightly modified setting, when we count the number of copies of $G$ with prescribed edge lengths. We choose $\frac{k}{2}$ points on a line $\ell$, Then, we fix a circle $C$ in a plane orthogonal to $\ell$, and centred in a point of $\ell$, and place $n-\frac{k}{2}$ points on it. Now it is easy to check that by picking the $\frac{k}{2}$ points from $\ell$, and any $\frac{k}{2}$ points from $C$, we obtain a copy of $G$.

\section{Concluding remarks and further problems}

While we could find sharp bounds for unit distance $k$-cycles for most $k$, on $\mathbb{R}^2$ and in $\mathbb{R}^3$ the problem seems more difficult. Proving good bounds for short cycles would be particularly interesting. In the plane, for $k=3$ an easy upper bound is $u_2(n)$, and the best lower bound is $ne^{\Omega(\frac{\log n}{\log \log n})}$. (See discussion in Chapter 6 of \cite{RPDG}.) For $k=4$ in the plane, we can construct $\Omega(u_2(n))$ many $4$-cycles by using two translated copies of a set with optimally many unit distances, and the best upper bound we could prove is $\tilde{O}(n^{\frac{5}{3}})$.

For $3$-regular graphs in $\mathbb{R}^3$, our bounds are sharp for bipartite graphs for the modified setting. It would be interesting to find sharp bounds in the general case, or at least for some small non-bipartite $3$-regular graphs.

\bibliographystyle{amsplain}
\bibliography{ref}

\providecommand{\bysame}{\leavevmode\hbox to3em{\hrulefill}\thinspace}
\providecommand{\MR}{\relax\ifhmode\unskip\space\fi MR }
\providecommand{\MRhref}[2]{%
  \href{http://www.ams.org/mathscinet-getitem?mr=#1}{#2}
}
\providecommand{\href}[2]{#2}
\begin{thebibliography}{10}

\bibitem{cuttinglemma}
P.~Agarwal and M.~Sharir, \emph{On the number of congruent simplices in a point
  set}, Discrete Comput. Geom \textbf{28} (2002), 123--150.

\bibitem{Brass}
P.~Brass, \emph{On the maximum number of unit distances among n points in
  dimension four}, Intuitive Geometry \textbf{6} (1997), 277--290.

\bibitem{RPDG}
P.~Brass, W.~O.~J. Moser, and J.~Pach, \emph{Research problems in discrete
  geometry}, Springer Science \& Business Media, 2006.

\bibitem{Szekely}
D.~De~Caen and L.~A. Sz{\'e}kely, \emph{On dense bipartite graphs of girth
  eight and upper bounds for certain configurations in planar point--line
  systems}, J. Combin. Theory Ser. A \textbf{77} (1997), no.~2, 268--278.

\bibitem{erdos1946sets}
P.~Erd{\H {o}}s, \emph{On sets of distances of n points}, Amer. Math. Monthly
  \textbf{53} (1946), no.~5, 248--250.

\bibitem{Lenz}
P.~Erdős, \emph{On sets of distances of n points in {E}uclidean space}, Magyar
  Tudom{\'a}nyos Akad{\'e}mia Matematikai Kutat{\'o} Int{\'e}zet Közleményei
  \textbf{5} (1960), 165--169.

\bibitem{almostsharp}
N.~Frankl and A.~Kupavskii, \emph{Almost sharp bounds on the number of discrete
  chains in the plane}, 36th International Symposium on Computational Geometry
  (SoCG 2020) (2020).

\bibitem{GPRS}
S.~Gunter, E.~Palsson, B.~Rhodes, and S.~Senger, \emph{Bounds on point
  configurations determined by distances and dot products}, arXiv preprint
  arXiv:2011.15055 (2020).

\bibitem{Kral}
P.~Klavík, D.~Král, and L.~Mach, \emph{Triangles in arrangements of points
  and lines in the plane}, J. Combin. Theory Ser. A \textbf{118} (2011),
  1140--1142.

\bibitem{discretechains}
E.~A. Palsson, S.~Senger, and A.~Sheffer, \emph{On the number of discrete
  chains}, Proc. Amer. Math. Soc. \textbf{149} (2021), no.~12, 5347--5358.

\bibitem{Passant}
J.~Passant, \emph{On {E}rd{\H{o}}s chains in the plane}, Bull. Korean Math.
  Soc. \textbf{58} (2021), no.~5, 1279--1300.

\bibitem{KMS}
S.~Senger S.~Kilmer, C.~Marshall, \emph{Dot product chains}, arXiv preprint
  arXiv:2006.11467 (2020).

\bibitem{4/3}
J.~Spencer, E.~Szemer\'{e}di, and W.~Trotter, Jr., \emph{Unit distances in the
  {E}uclidean plane}, Graph theory and combinatorics ({C}ambridge, 1983),
  Academic Press, London, 1984, pp.~293--303. \MR{777185}

\bibitem{Swanepoel}
K.~J. Swanepoel, \emph{Unit distances and diameters in euclidean spaces},
  Discrete Comput. Geom. \textbf{41} (2009), no.~1, 1--27.

\bibitem{SzT}
E.~Szemer{\'e}di and W.~T. Trotter, \emph{Extremal problems in discrete
  geometry}, Combinatorica \textbf{3} (1983), no.~3, 381--392.

\end{thebibliography}

\end{document}